\documentclass[11pt]{amsart}

\usepackage{amsmath,amsthm,amssymb,mathrsfs,amsfonts,verbatim,enumitem,color,leftidx}
\usepackage{mathabx}
\usepackage{etoolbox, tikz} 
\usepackage{bbm}
\usepackage[all,tips]{xy}
\usepackage{graphicx,ifpdf}
\usepackage{here}
\ifpdf
\fi
\usepackage{hyperref}
\usepackage[right,displaymath,mathlines]{lineno}
\allowdisplaybreaks[0]
\newtheorem{thm}{Theorem}[section]
\newtheorem{lem}[thm]{Lemma}
\newtheorem{coro}[thm]{Corollary}
\newtheorem{prop}[thm]{Proposition}

\theoremstyle{definition}
\newtheorem{de}[thm]{Definition}

\theoremstyle{remark}

\numberwithin{equation}{section}

\makeatletter

\newcommand{\Rmnum}[1]{\expandafter\@slowromancap\romannumeral #1@}
\makeatother

\newcommand{\Z}{\mathbb{Z}}

\newcommand{\cT}{\mathcal{T}}

\newcommand{\cX}{\mathcal{X}}

\begin{document}
\date{January 28, 2019}

\title[Limiting distribution of geodesics]{Limiting distribution of geodesics in a geometrically finite quotients of regular trees}
\author{\scshape Sanghoon Kwon} 
\author{\scshape Seonhee Lim}
\maketitle{}

% Information for first author
%\author{}

%\author{Sanghoon Kwon}
% Address of record for the research reported here
%\address{}
% Current address
%\curraddr{}
% \email{skwon@kias.re.kr}
% \thanks will become a 1st page footnote.
\thanks{}

%\author{}

%\dedicatory{This paper is dedicated to my advisor.}

%\keywords{homogeneous dynamics,  equidistribution, ergodic theorem}

\begin{abstract}
In this article, we prove an extreme value theorem on the limit distribution of geodesics in a geometrically finite quotient of $\Gamma\backslash\mathcal{T}$ a locally finite tree. Main examples of such graphs are quotients of a Bruhat-Tits tree $\mathcal{T}$ by non-cocompact discrete subgroups $\Gamma$ of $PGL(2,\mathbf{K})$ of a positive characteristic local field $\mathbf{K}$. We investigate, for a given time $T$, the measure of the set of $\Gamma$-equivalent geodesic classes which stay up to time $T$ the region of distance $d$ at most $N$ depending on $T$ from a fixed compact subset $D$ of $\Gamma\backslash\mathcal{T}$. Namely, for Bowen-Margulis measure $\mu$ on the space $\Gamma\backslash\mathcal{GT}$ of geodesics and the critical exponent $\delta$ of $\Gamma$, we show that there exists a constant $C$ depending on $\Gamma$ and $D$ such that $$\lim_{T\to\infty}\mu\left(\left\{[l]\in\Gamma\backslash\mathcal{GT}\colon \underset{0\le t \le T}{\textrm{max}}d(D,l(t))\le N+y\right\}\right)=e^{-q^y/e^{2\delta y}}$$ with $$N=\log_{e^{2\delta/q}}\left(\frac{T(e^{2\delta-q)}}{2e^{2\delta}-C(e^{2\delta}-q)}\right).$$
\end{abstract} 

%\begin{abstract}

%\end{abstract}

%\maketitle

\markright{}
\section{Introduction}\label{sec:1}
Extreme value theory in probability theory has been developed for the last several decades. Whenever there is a random process, one can consider probabilistic questions such as central limit theorem, local limit theorem, law of large number, etc. One of the probabilistic questions is the extreme value distribution: given an $N$, either a constant or a function of time $T$, 
what is the probability that your given process up to time $T$ is of distance at most $N$ as $T$ tends to infinity? 

Recently, there has been a series of results on stationary stochastic processes arising from various chaotic dynamical systems such as random walks starting from \cite{Co} (see a survey paper \cite{Fr} and references therein).

One can ask a similar question for geodesics in a non-compact manifold with respect to a measure on the set of geodesics: 
\begin{center}
{\it What is the measure of the set of geodesic rays visiting \\
the region of distance at most $N$  from a fixed point in time $[0,T]$?
}
\end{center}
For the modular surface $\mathbb{H}^2/SL_2(\Z)$, a related question on continued fraction expansion was answered by Galambos \cite{Ga} whose result was used by Pollicott for the question on geodesics \cite{Po}.

In this article, we address the same question for quotient graphs of locally finite trees, which are the non-Archimedean analog of hyperbolic surfaces. We obtain an extreme value distribution for geometrically finite quotients of certain locally finite trees, in particular, regular trees. These include all the algebraic quotients of Bruhat-Tits tree of the group $PGL_2$ over positive characteristic local fields.

Let us state our main result. 
Given a locally finite tree $\mathcal{T}$, let $\textrm{Aut}(\mathcal{T})$ be the group of automorphisms of $\mathcal{T}$ and $\Gamma$ be a geometrically finite discrete subgroup of $\textrm{Aut}(\mathcal{T})$. Suppose that the quotient graph $\Gamma\backslash\mathcal{T}_{\textrm{min}}$ of the minimal $\Gamma$-invariant subtree $\mathcal T_{min}$ is a union of a finite graph with finitely many rays each of which is a ray of Nagao type as in Figure~\ref{fig:1} (see \cite{p04}), with rays each of which has \textit{edge-index alternating between $q$ and 1}. Main examples are geometrically finite discrete subgroups of a $(q+1)$-regular tree.  We denote by $\delta=\delta_\Gamma$ the critical exponent of $\Gamma$, which is defined by $$\delta=\delta_\Gamma=\underset{n\to\infty}{\overline{\lim}}\frac{\log\#\{\gamma\in\Gamma\colon d(x,\gamma x)\le n\}}{n}$$ for any fixed vertex $x\in V\mathcal{T}$. The value does not depend on the choice of $x\in V\mathcal{T}$.

\vspace{0.5em}
\begin{figure}[H]\label{fig:1}
\begin{center}
\begin{tikzpicture}[every loop/.style={}]
  \tikzstyle{every node}=[inner sep=0pt]
  \node (0) {} node [above=4pt] at (0.5,1.5) {};
  \node (2) at (1.5,0) {$\bullet$} node [below=4pt] at (1.5,0) {$$}; 
  \node (4) at (3,0) {$\bullet$}node [below=4pt] at (3,0) {$$}; 
  \node (6) at (4.5,0) {$\bullet$}node [below=4pt] at (4.5,0) {$$}; 
  \node (8) at (6,0) {$\cdots$}node [below=4pt] at (6,0) {$$}; 
%  \node (10) at (6,0) {$\cdots$}node [below=4pt] at (6,0) {$$}; 
  
  \path[-] %(0) edge node [above=4pt] { a \,\, $q$} (2)
 (2) edge node [above=4pt] {$1$ \quad \,\,$q$} (4)
(4) edge node [above=4pt] {$1$ \quad \,\,\,\,$q$} (6)
 (6) edge node [above=4pt] {$1$ \quad \,\,\,\,$q$} (8);
 %(8) edge node [above=4pt] {$1$ \quad \,\,\,\,$q$} (10);
 
%   \node (10) {$\bullet$} node [below=4pt] at (0,1) {$$};
  \node (20) at (1.5,1) {$\bullet$} node [below=4pt] at (1.5,1) {$$}; 
  \node (40) at (3,1.1) {$\bullet$}node [below=4pt] at (3,1.1) {$$}; 
  \node (60) at (4.5,1.2) {$\bullet$}node [below=4pt] at (4.5,1.2) {$$}; 
  \node (80) at (6,1.3) {$\cdots$}node [below=4pt] at (6,1.3) {$$}; 
 % \node (100) at (7.5,1.4) {$\cdots$}node [below=4pt] at (7.5,1.4) {$$}; 
  
  \path[-] %(0) edge node [above=4pt] { } (10)
  (2) edge node [above=4pt] {  $$} (20)
 (20) edge node [above=4pt] {$1$ \quad \,\,$q$} (40)
(40) edge node [above=4pt] {$1$ \quad \,\,\,\,$q$} (60)
 (60) edge node [above=4pt] {$1$ \quad \,\,\,\,$q$} (80);
% (80) edge node [above=4pt] {$1$ \quad \,\,\,\,$q$} (100);

  % \node (1) {$\bullet$} node [below=4pt] at (0,0) {$$};
  \node (3) at (-1.5,0) {$\bullet$} node [below=4pt] at (-1.5,0) {$$}; 
  \node (5) at (-3,0) {$\bullet$}node [below=4pt] at (-3,0) {$$}; 
  \node (7) at (-4.5,0) {$\bullet$}node [below=4pt] at (-4.5,0) {$$}; 
  \node (9) at (-6,0) {$\cdots$}node [below=4pt] at (-6,0) {$$}; 
%  \node (11) at (-7.5,0) {$\cdots$}node [below=4pt] at (-7.5,0) {$$}; 

   \path[-] 
(3) edge node [above=4pt] {$q$ \quad \,\,$1$} (5)
(5) edge node [above=4pt] {$q$ \quad \,\,\,\,$1$} (7)
 (7) edge node [above=4pt] {$q$ \quad \,\,\,\,$1$} (9);
% (9) edge node [above=4pt] {$q$ \quad \,\,\,\,$1$} (11);

%   \node %(10) {$\bullet$} node [below=4pt] at (0,1) {$$};
  \node (30) at (-1.5,1) {$\bullet$} node [below=4pt] at (-1.5,1) {$$}; 
  \node (50) at (-3,1.1) {$\bullet$}node [below=4pt] at (-3,1.1) {$$}; 
  \node (70) at (-4.5,1.2) {$\bullet$}node [below=4pt] at (-4.5,1.2) {$$}; 
  \node (90) at (-6,1.3) {$\cdots$}node [below=4pt] at (-6,1.3) {$$}; 
 % \node (110) at (-7.5,1.4) {$\cdots$}node [below=4pt] at (-7.5,1.4) {$$}; 

   \path[-] 
  % (0) edge node [above=4pt] {??} (10)
  (3) edge node [above=4pt] {$ $} (30)
  (3) edge node [above=4pt] {$ $} (2)
  (20) edge node [above=4pt] {$ $} (2)
  (30) edge node [below=8pt] {D} (20) 
 (30) edge node [above=4pt] {$q$ \quad \,\,$1$} (50)
(50) edge node [above=4pt] {$q$ \quad \,\,\,\,$1$} (70)
 (70) edge node [above=4pt] {$q$ \quad \,\,\,\,$1$} (90);
% (90) edge node [above=4pt] {$q$ \quad \,\,\,\,$1$} (110);

\end{tikzpicture}
\vspace{0.5em}
\caption{The quotient graph of a geometrically finite subgroup with compact part $D$}
\end{center}
\end{figure}

Let $h_T^{(l)}$ be the maximum of the \emph{height} of $l$ among $t\in[0,T]$, which is the distance from the compact part, say $D$, in Figure~\ref{fig:1}:
$$ h_T^{(l)} = \max_{0 \leq t \leq T} d(D, l(t)).$$
Let $\mu$ be the Bowen-Margulis measure. (See Definition~\ref{def:BM}.)
\begin{thm} Let $\Gamma$ be a geometrically finite discrete subgroup of $\textrm{Aut}(\mathcal{T})$ of a $(q+1)$-regular tree $\mathcal T$. There exists a constant $C=C(\Gamma)$ such that
$$\lim_{T\to\infty}\mu\left(\left\{[l]\in\Gamma\backslash\mathcal{GT}|h_T^{(l)}\le N+y\right\}\right)=e^{-q^y/e^{2\delta y}},$$
where $$N=\log_{e^{2\delta}/q}\left(\frac{T(e^{2\delta}-q)}{2e^{2\delta}-C(e^{2\delta}-q)}\right).$$
\end{thm}
%\section{Limiting distribution of geodesics in a geometrically finite quotient of regular graphs}

If $\Gamma$ is a lattice subgroup of $\textrm{Aut}(\mathcal{T})$, then $\delta=\log q$ and $\mu$ is $Aut(\mathcal T)$-invariant. Moreover, if the quotient itself is a ray of Nagao type, then the constant $C$ is equal to 0. (See the proof in Section~\ref{sec:2}). The main example of $\Gamma$ is $PGL_2\left(\mathbb{F}_q[t]\right)$ sitting in $PGL_2\left(\mathbb{F}_q(\!(t^{-1})\!)\right)$ (see Section~\ref{sec:2}).  This yields the following corollary.
\begin{coro} Suppose that $\Gamma$ is a discrete subgroup of $\textrm{Aut}(\mathcal{T})$ such that the edge-indexed graph associated to $\mathcal{T}//\Gamma$ is equal to the ray $\cX$ of Nagao type. Then we have
$$\lim_{T\to\infty}\mu\left(\left\{[l]\in\Gamma\backslash\mathcal{GT}|h_{T}^{(l)}\le N+y\right\}\right)=e^{-1/q^y}$$
with $$N=\log_q\left(\frac{T(q-1)}{2q}\right).$$
\end{coro}

We remark that for quotient spaces of lattices in Lie groups, Kirsebom \cite{Ki} showed some estimates for the limiting distribution of the maximum height over a specific interval of indices with respect to certain sparse
subsequences of the one-parameter action.
%\color{blue}Please check if this part is okay.\color{black}

The article is organized as follows. In Section~\ref{sec:3}, we review the Markov chain associated to the discrete time geodesic flow of edge-indexed graphs and the construction of Gibbs measures. In Section~\ref{sec:2}, we prove the extreme value distribution for the simplest case, the ray of Nagao type. We prove the extreme value distribution of geometrically finite quotients in Section~\ref{sec:4} using the theory of countable Markov chain and the result for the ray of Nagao type. We tried to write Section~\ref{sec:2} as self-contained as possible (without Markov chain) for the readers who are mainly interested in the modular ray.

\section{ Markov chain and Gibbs measures }\label{sec:3}
As in the introduction, let $\mathcal T$ be a locally finite tree and $\Gamma$ a discrete subgroup of $\textrm{Aut}(\mathcal{T})$. In this section, we do not need any assumption on the indices of the quotient graph $\Gamma\backslash\mathcal{T}_{\textrm{min}}$. Let $\delta=\delta_\Gamma$ be the critical exponent of $\Gamma$, which is defined by $$\delta=\delta_\Gamma=\underset{n\to\infty}{\overline{\lim}}\frac{\log\#\{\gamma\in\Gamma\colon d(x,\gamma x)\le n\}}{n}$$ for any fixed vertex $x\in V\mathcal{T}$. The value does not depend on the choice of $x\in V\mathcal{T}$.

\subsection{Bowen-Margulis measure}
In this subsection, we review the construction of geodesic flow invariant measure $\mu$ on the space of bi-infinite geodesics associated to a conformal family $\{\mu_x\}_{x\in V\mathcal{T}}$ of measures on the boundary $\partial_\infty\mathcal{T}$ at infinity. Such a measure $\mu$ is finite when $\Gamma$ is geometrically finite. For lattices of Nagao type, it coincides with Haar measure coming from $Aut(\mathcal T)$. 

The construction is similar to the construction of Bowen-Margulis measure from Patterson-Sullivan density, more generally that of Gibbs measures from conformal densities. (See \cite{Sull} for hyperbolic manifolds, \cite{Rob} for CAT(-1) spaces, and \cite{BPP} for trees.)

Let us fix a vertex $x \in V\cT$.
Let $\mathcal{GT}=\{l\colon \mathbb{Z}\to V\mathcal{T}, n\mapsto l_n\textrm{ isometry}\}$ be the space of bi-infinite geodesics and $\phi$ the discrete time geodesic flow on $\mathcal{GT}$ given by $\phi(l)(n)=l(n+1)$. Let $\mathcal{GT}^+=\{ l\colon \mathbb{Z}_{\ge0} \to V\cT, n \mapsto l_n \textrm{ isometry}\}$ be the space of geodesic rays and $\mathcal{GT}^+_x=\{ l \in\mathcal{GT}^+\,|\, l_0=x \}$ be the space of geodesic rays starting at $x$. 

Let $\partial_\infty\mathcal{T}$ be the Gromov boundary at infinity of $\mathcal{T}$. For a fixed a vertex $x\in V\mathcal{T}$, the Gromov boundary $\partial_\infty\mathcal{T}$ can be identified with $\mathcal{GT}_x^+$.

Let $\pi\colon\mathcal{T}\to\Gamma\backslash\mathcal{T}$ be the natural projection. It induces the natural projection map $\mathcal{GT}\to\Gamma\backslash\mathcal{GT}$ which we will also denote by $\pi$.

\begin{de}[Patterson-Sullivan density]\label{def:3.1}  Given $\omega\in\partial_\infty\mathcal{T}$ and $x,y\in V\mathcal{T}$, the Busemann cocycle $\beta_\omega(x,y)$ is defined as $d(x,z)-d(y,z)$ where $[x,\omega)\cap[y,\omega)=[z,\omega)$.
\begin{enumerate}
\item A \emph{Patterson density} of dimension $\delta$ for a discrete group $\Gamma<\textrm{Aut}(\mathcal{T})$ is a family of finite nonzero positive Borel measures $\{\mu_x\}_{x\in V\mathcal{T}}$ on $\partial_\infty\mathcal{T}$ such that for every $\gamma\in\Gamma$, for all $x,y\in\mathcal{T}$ and $\omega\in\partial_\infty\mathcal{T}$, 
$$\gamma_*\mu_x=\mu_{\gamma\cdot x}\qquad\textrm{and}\qquad \frac{d\mu_x}{d\mu_y}(\omega)=e^{-\delta\beta_\omega(x,y)}.$$
\item For $\Gamma$ geometrically finite, the Patterson density $\{\mu_x\}_{x\in V\mathcal{T}}$ of dimension $\delta=\delta_\Gamma$ is the (unique) weak-limit of $\mu_{x,s}$ as $s\to \delta^+$ where 
$$\mu_{x,s}=\frac{1}{\sum_{\gamma\in\Gamma}e^{-sd(s,\gamma x)}}\sum_{\gamma\in\Gamma}e^{-sd(x,\gamma x)}\delta_{\gamma x}$$ and $\delta_{\gamma x}$ is the Dirac mass at $\gamma x$ (\cite{hp07}). 
\end{enumerate}
\end{de} 

Now consider the set $\mathcal{GT}_x$ of bi-infinite geodesics which reaches $x$ at time zero. On the set $\mathcal{GT}_x$, we
define $\mu$ locally by $\mu_x \times \mu_x$:  for $D^-,D^+\subset\partial_\infty\mathcal{T}$ such that every geodesic line connecting a point in $D^-$ and $D^+$ passes through $x$, we define $$(\mu_x\times\mu_x)\left(\{ l \in \mathcal{GT}_x : l^-\in D^-\textrm{ and }l^+\in D^+\}\right)=C_x \mu_x(D^-)\mu_x(D^+)$$
on $\mathcal{GT}_x$. Here, $C_x$ is the normalizing constant such that $(\mu_x\times\mu_x)(\mathcal{GT}_x)=1$. 

Now we use a ramified covering argument: since there is a one-to-one correspondence between $\Gamma \backslash \mathcal{GT}$ and $\underset{[x]\in 
\Gamma\backslash V\mathcal{T}}{\coprod} \Gamma_x \backslash \mathcal{GT}_x$, take the sum of $\mu_x \times \mu_x$ and normalize.
\begin{de}[Bowen-Margulis measure]\label{def:BM}
For a measurable subset $E\subset \Gamma\backslash \mathcal{GT}$, define $\mu(E)$ to be 
$$\mu(E):=C_0 \sum_{[x]\in \Gamma\backslash V\mathcal{T}}\frac{1}{|\Gamma_x|}(\mu_x \times \mu_x)(\pi^{-1}E\cap \mathcal{GT}_x),$$ where $C_0=(\underset{[x]}{\sum}\frac{1}{|\Gamma_x|})^{-1}.$ 
\end{de}
Note that $C_0$ is chosen so that $\mu(\Gamma\backslash\mathcal{GT})=1$ and the quantity above is well-defined i.e. it depend only on the class of $x$. The measure $\mu$ is $\phi$-invariant. Indeed, any set can be decomposed into projection of cylinders of the form 
$E = \pi (C_E)$ with $C_E =\{ l \in \mathcal{GT}_x : l^- \in D^-, l^+ \in D^+ \}$ small enough so that $\pi$ is one-to-one on $C_E$. Note that the measure of such cylinders are $\phi$-invariant:
$$\mu (\phi^{-1} E) =  \frac{1}{|\Gamma_{x'}|} (\mu_{x'} \times \mu_{x'} )( \pi^{-1} \phi^{-1} E \cap \mathcal{GT}_{x'}) 
= \frac{1}{|\Gamma_{x'}|} |\Gamma_{x'}|(\mu_{x'} \times \mu_{x'} )(C_{ \phi^{-1} E} \cap \mathcal{GT}_{x'}) 
$$
$$ =  \frac{1}{|\Gamma_{x}|} |\Gamma_x|( \mu_{x} \times \mu_{x}) (C_E \cap \mathcal{GT}_{x}) = \frac{1}{|\Gamma_{x}|}(\mu_{x} \times \mu_{x}) (\pi^{-1} E \cap \mathcal{GT}_{x}) = \mu(E),$$
where $x'$ is the base point of the elements of $\phi^{-1}E$.
The third equality is by definition of $\mu_x \times \mu_x$.
Another way of seeing $\phi$-invariance is to observe that $\mu$ is a Gibbs measure for dicrete time geodesic flow $\phi$. Compared with Proposition 4.13 of \cite{BPP}.

\subsection{Markov chain of $\Gamma_f \backslash \mathcal{GT}$}
In this subsection, we explain a way to obtain Markov chain associated to the geodesic flow on the compact part. First enlarge the given geometrically finite group $\Gamma$ to the full group $\Gamma_f$ associated to $\Gamma$, which is defined as the group maximal with the property that the quotient graph $\Gamma \backslash \cT$ coincides with $\Gamma_f \backslash \cT$ \cite{bm96}, namely $$\Gamma_f=\{g\in\textrm{Aut}(\mathcal{T})\,|\,\pi\circ g=\pi\}.$$
Note that $\Gamma \backslash \mathcal {GT}$ and $\Gamma_f \backslash \mathcal {GT}$ can be very different.
 We will define a Markov chain of $\Gamma_f \backslash \cT$ coding the geodesic flow. We remark that $\Gamma_f$ is not necessarily virtually discrete, thus, the Markov chain of $\Gamma_f \backslash \cT$ does not necessarily give a Markov chain of $\Gamma \backslash \cT$ coding the geodesic flow. However, the quotient graphs are identical, thus the extreme value condition for $\Gamma_f$ holds if and only if the same condition holds for $\Gamma$ if we consider the measure on $\Gamma \backslash \mathcal {GT}$ induced from $\Gamma_f \backslash \mathcal {GT}$.

More precisely, let $\mu$ be the Bowen-Margulis measure defined in Definition~\ref{def:BM}. Denote by $p$ the natural projection $\Gamma\backslash\mathcal{GT}\to\Gamma_f\backslash\mathcal{GT}$ such that $\phi \circ p = p \circ \phi$. It is an important fact that the set $\left\{[l]\in\Gamma\backslash\mathcal{GT}\colon h_{T}^{(l)}\le N+y\right\}$ is invariant under the associated full group $\Gamma_f$. Thus, if we denote by $\overline{\mu}$ the measure on $\Gamma_f\backslash\mathcal{GT}$ given by $\overline{\mu}(E)=\mu(p^{-1}(E))$, then it suffices to consider the limiting distribution of $$\overline{\mu}\left(\left\{[l]\in\Gamma\backslash\mathcal{GT}_f\colon h_{T}^{(l)}\le N+y\right\}\right).$$

Now we introduce the Markov chain associated to the discrete time geodesic flow. 

Given an undirected graph $A$, let $EA$ the set of all oriented edges where every edge of $A$ is bi-directed. The cardinality of $EA$ is twice the number of edges of $A$. For $e\in EA$, let $\partial_0e$ and $\partial_1e$ be the initial vertex and the terminal vertex of $e$, respectively. Denote by $\overline{e}$ be the opposite edge of $e$ satisfying $\partial_i\overline{e}=\partial_{1-i}e$ for $i=0,1$. An \emph{edge-indexed graph} $(A,i)$ is a bi-directed graph $A$ together with a map $i\colon EA\to\mathbb{Z}_{\ge 0}$ assigning a positive integer to each oriented edge.

For a given edge-indexed graph $(A,i)$, consider the following subset $$X_{(A,i)}=\{x=(e_j)_{j\in\mathbb{Z}}\,|\,\partial_0e_{j+1}=\partial_1e_j\textrm{ and if } e_{j+1}=\overline{e_j},\textrm{ then }i^A(e_j)>1\}$$ of  
admissible paths in $(EA)^{\mathbb{Z}}$. The family of \emph{cylinders} 
$$[e_0,\cdots,e_{n-1}]:=\{x\in X_{(A,i)}\colon x_i=e_i,i=0,\cdots,n-1\}$$ is a basis of open sets for a topology on $X$. Let $\sigma\colon X_{(A,i)}\to X_{(A,i)}$ be the shift given by $\sigma(x)_i:=x_{i+1}$. Then $(\mathcal{GT},\phi)$ is conjugate to $(X_\mathcal{T},\sigma)$ (Consider $\mathcal{T}$ as $(\mathcal{T},i^0)$ with $i^0(e)=1,\forall e\in E\mathcal{T}$). If $(A,i)$ is the edge-indexed graph associated with the quotient graph of groups $\Gamma_f\backslash\!\backslash\mathcal{T}$ (see \cite{Se} for the definition of quotient graphs of groups), then we also have a bijection $\Phi\colon(\Gamma_f\backslash \mathcal{GT},\phi)\to( X_{(A,i)},\sigma)$ given by $\Phi([l])=(e_j)_{j\in\mathbb{Z}}$, $\partial_ie_j=l_{i+j}$ for all $j\in\mathbb{Z}$ and $i=0,1$, so that the following diagram commute (cf. \cite{bm96}).
\begin{equation*}
\begin{aligned}\label{1}
\xymatrix{ \Gamma\backslash \mathcal{GT} \ar[r]^{\,\, p\quad} \ar[d]_{\phi} & \Gamma_f\backslash \mathcal{GT} \ar[r]^{\Phi}\ar[d]_{\phi} & X_{(A,i)} \ar[d]^{\sigma}  \\ 
\Gamma\backslash \mathcal{GT} \ar[r]^{\,\,p\quad} & \Gamma_f\backslash \mathcal{GT} \ar[r]^{\Phi} & X_{(A,i)} \\
}
\end{aligned}
\end{equation*}
%\begin{equation}
%\begin{aligned}\label{1}
%\xymatrix{ \Gamma\backslash \mathcal{GT} \ar[r]^{\quad\phi\quad} \ar[d]_{p} & \Gamma\backslash \mathcal{GT} \ar[d]_{p} \\ 
%\Gamma_f\backslash \mathcal{GT} \ar[r]^{\quad\phi\quad} \ar[d]_{\Phi} & \Gamma_f\backslash \mathcal{GT} \ar[d]_{\Phi} \\
% X_{(A,i)} \ar[r]^{\sigma}  & X_{(A,i)} }
%\end{aligned}
%\end{equation}
%\seon{Two squares here!}

For $f\in E\mathcal{T}$, we denote the \emph{shadow} of an edge $f$ by $$\mathcal{O}(f)=\{\omega\in\partial_\infty\mathcal{T}\,|\, \exists \xi\in \mathcal{GT}\textrm{ such that }\xi_0=\partial_0f,\xi_1=\partial_1f\textrm{ and }\xi^+=\omega\}.$$ Let $[e_0,\ldots,e_{n-1}]$ be an admissible cylinder of $X_{(A,i)}$. Following \cite{bm96}, we define $\lambda$ by
$$\lambda([e_0,\cdots,e_{n-1}])=\frac{\mu_{\partial_0f_0}(\mathcal{O}(\overline{f_0}))\mu_{\partial_1f_{n-1}}(\mathcal{O}(f_{n-1}))}{|\Gamma_{f_0,\cdots,f_{n-1}}|}e^{-n\delta}$$ where $f_j$ is an oriented edge of $\mathcal{T}$ for which $\pi(f_j)=e_j$ and $\partial_1f_j=\partial_0f_{j+1}$ and $\Gamma_{f_0,\cdots,f_{n-1}}$ is the stabilizer group of $f_0,\ldots, f_{n-1}$ of $\Gamma$. This quantity does not depend on the choice of $f_j$.

 It has the Markov property, namely
\begin{equation}\label{2}
\sum_{e_j\colon\partial_0e_k=\partial_1e_j}\lambda([e_j,e_k]) =\lambda([e_k]), \;\; \mathrm{and} \;\;
\sum_{e_k\colon\partial_0e_k=\partial_1e_j}\lambda([e_j,e_k]) =\lambda([e_j]).
\end{equation}
 We also have
\begin{align}\label{3}
\int_{\Gamma_f\backslash \mathcal{GT}}f \,d\overline{\mu}= \int_{X_{(A,i)}}\Phi^*(f)\,d\lambda
\end{align}
where $\Phi^*(f)[(x_i)_{i\in\mathbb{Z}}]=f[\Phi^{-1}((x_i)_{i\in\mathbb{Z}})].$
In other words, the measure $\lambda$ is a Markov measure and two dynamical systems $(\Gamma_f\backslash \mathcal{GT},\phi,\overline{\mu})\textrm{ \ and \ }(X_{(A,i)},\sigma,\lambda)$ are isomorphic (\cite{bm96}).
%\seon{Read carefully here}

Let us briefly recall positively recurrent Markov chain following \cite{MT}. Let $Z_n$ be a Markov chain with phase space $\mathcal{S}=\{s_1,s_2,\cdots\}$ and transition probabilities $$p_{ij}=p_{s_is_j}=P\{Z_{n+1}=s_j\,|\,Z_n=s_i\}, \qquad \sum_j p_{ij} =1.$$ For a subset $B\subset \mathcal{S}$ of alphabets, let 
\begin{equation}
\begin{aligned}\label{4}
f_{ij}^{(n)} &=P\{Z_1\ne  s_j,\cdots, Z_{n-1}\ne s_j, Z_n=s_j\,|Z_0=s_i\}, \\
p_{ij}^{(n)} &=P\{ Z_n=s_j\,|Z_0=s_i\}.
\end{aligned}
\end{equation}
and set $f_{ij}^{(0)}=0$ and $p_{ij}^{(0)}=\delta_{ij}$. Observe the following convolution relation \begin{align}\label{conv} p_{s_is_j}^{(n)}=\sum_{r=1}^{n}f_{s_is_i}^{(r)}p_{s_is_j}^{(n-r)}.
\end{align}
Suppose that the Markov chain $Z_n$ is irreducible, i.e., for any $s_i,s_j\in\mathcal{S}$, there exists $n>0$ such that $p_{ij}^{(n)}>0$. We say $\pi_j$ is a \emph{stationary} distribution if it satisfies $\pi_j=\sum_{i\in\mathcal{S}}\pi_i p_{ij}$. A Markov chain $(\mathcal{S},p_{ij})$ is \emph{recurrent} if a stationary distribution exists and furthermore it is called \emph{positive recurrent} if $\sum_{n=1}^{\infty}nf_{jj}^{(n)}<\infty$. When $(\mathcal{S},p_{ij},\pi_j)$ is positive recurrent, $\pi_j$ is unique and we have $$\pi_j=\frac{1}{\sum_{n=1}^{\infty}nf_{jj}^{(n)}}.$$
 An irreducible Markov chain is called \emph{aperiodic} if for some (and hence every) state $s_i\in \mathcal{S}$, its period $\textrm{gcd}\{n\colon p_{ii}^{(n)}>0\}$ is 1.

The Markov chain we consider in this article is (see \cite{k18}) $(\mathcal{S},p_{ij})$ for \begin{equation}\label{MCdef}\mathcal{S}=E(\Gamma\backslash\!\backslash\mathcal{T}),\quad p_{ij}=p_{e_ie_j}=\frac{\lambda([e_i,e_j])}{\lambda([e_i])}.\end{equation} Note that the stationary distribution is $ \pi_j=\lambda([e_j])$
which is positive recurrent and aperiodic. If a positive recurrent Markov chain $Z_n$ is aperiodic, then $\pi_j=\underset{n\to\infty}{\lim}p_{ij}^{(n)}$ and $\pi_j$ does not depend on the choice of $i\in\mathcal{S}$ (Chapter 8-10, \cite{MT}). We will use this fact in Section~\ref{sec:4}.

\section{Extreme value distribution for rays of Nagao type}\label{sec:2}
Let $\mathcal{T}$ be the $(q+1)$-regular tree and let $G=Aut(\cT)$. Denote by $V\mathcal{T}$ the set of vertices of $\mathcal{T}$. Let $\mathcal{X}$ be the edge-indexed graph described in Figure~\ref{fig:2} and $\Gamma$ be the fundamental group of a finite grouping of $\cX$. In other words, $\Gamma$ is a discrete subgroup of $G$ for which the edge-indexed graph associated to the quotient graph of groups $\Gamma\backslash\!\backslash\mathcal{T}$ is $\cX$. Let us denote the vertices of $\mathcal{X}$ by $v_0, v_1, v_2, \dots$ as in Figure~\ref{fig:2}. 
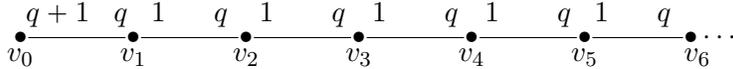
\begin{figure}[H]
\begin{center}
\begin{tikzpicture}[every loop/.style={}]
  \tikzstyle{every node}=[inner sep=0pt]
  \node (0) {$\bullet$} node [below=4pt] at (0,0) {$v_0$};
  \node (2) at (1.5,0) {$\bullet$} node [below=4pt] at (1.5,0) {$v_1$}; 
  \node (4) at (3,0) {$\bullet$}node [below=4pt] at (3,0) {$v_2$}; 
  \node (6) at (4.5,0) {$\bullet$}node [below=4pt] at (4.5,0) {$v_3$}; 
  \node (8) at (6,0) {$\bullet$}node [below=4pt] at (6,0) {$v_4$}; 
  \node (10) at (7.5,0) {$\bullet$}node [below=4pt] at (7.5,0) {$v_5$}; 
  \node (11) at (9.2,0) {$\bullet \cdots$}node [below=4pt] at (9,0) {$v_6$}; 

  \path[-] (0) edge node [above=4pt] {$q+1$ \,  $q$} (2)
 (2) edge node [above=4pt] {$1$ \quad \,\,$q$} (4)
(4) edge node [above=4pt] {$1$ \quad \,\,\,\,$q$} (6)
 (6) edge node [above=4pt] {$1$ \quad \,\,\,\,$q$} (8)
 (8) edge node [above=4pt] {$1$ \quad \,\,\,\,$q$} (10)
 (10) edge node [above=4pt] {$1$ \quad \,\,$q$ \ } (11);
\end{tikzpicture}
\label{fig:2}
\vspace{0.5em}
\caption{A ray of Nagao type}
\end{center}
\end{figure}
The main motivating example is the modular ray: let $\mathbf{K}=\mathbb{F}_q(\!(t^{-1})\!)$ and $\mathbf{Z}=\mathbb{F}_q[t]$. Let $G=PGL(2,\mathbf{K})$ and $\Gamma=PGL(2,\mathbf{Z})$. The group $G$ acts transitively on the $(q+1)$-regular tree $\cT$ which is called the Bruhat-Tits tree associated to $G$ \cite{Se}. 
Let us fix a vertex $x \in V\cT$.

Recall that $\mathcal{GT}, \mathcal{GT}^+, \mathcal{GT}^+_x$ are the space of biinfinite geodesics, geodesic rays and geodesic rays starting at $x$, respectively. 

Let $\pi\colon\mathcal{T}\to\Gamma\backslash\mathcal{T}$ be the natural projection. It induces the natural projection map $\mathcal{GT}\to\Gamma\backslash\mathcal{GT}$ which we will also denote by $\pi$.
\begin{de}\label{excursion} Let us denote $ \pi^{-1}v_0 \cap \{ l_n : n \in \mathbb{Z}_{\ge 0} \} = \{l_{t_1}, l_{t_2}, \cdots \}.$ For a fixed geodesic $l$, such $t_i$ will be denoted by $t_i(l)$.
\begin{enumerate} 
\item The sequence of vertices $l_{t_{n}}l_{t_{n}+1} \cdots l_{t_{n+1}}$ is called the \emph{$n$-th excursion of $l$}. For given $l$, such $t_i$ will be denoted by $t_i{(l)}$.
\item Since the quotient graph of groups has indices alternating between $q$ and $1$ as in Figure~\ref{fig:2}, all the neighbors of $\pi^{-1}(v_0)$ are mapped to $v_1$ under $\pi$. As for $i \geq 1$, all but one neighbors of $\pi^{-1}(v_i)$ are mapped to $v_{i-1}$ under $\pi$ and the remaining one is mapped to  $v_{i+1}$.  As any geodesic has no back-tracking, any $n$-th excursion of geodesic
$l_{t_{n-1} } \cdots l_{t_n}$ projects to $v_0 \cdots v_m v_{m-1} \cdots v_1 v_0$.
%The height $a_n^{(l)}$ of $n$-th excursion of $l$ in $\Gamma\backslash\mathcal{T}$ is defined as the $n$-th local maximum of the distance of $l$ from $v_0$, which is $N$ in the above expression. 
Call such $m$ the \emph{height} of $n$-th excursion of $l$ and denote it by $a_n{(l)}$.
The \textit{$n$-th excursion time} is $$ t_{i+1}{(l)} - t_i{(l)} =2 a_i{(l)}.$$
\end{enumerate}
\end{de}
\begin{de}
Let $\mu_x$ be the probability measure on $\mathcal{GT}^+_x$ defined as follows. The subsets $E_y=\{l\in \mathcal{GT}_x^+\colon l\textrm{ passes through }y\}$ for $y \in\mathcal{T}$  form a basis for a topology on $\mathcal{GT}_x^+$. 
Let $\mathcal B$ be the associated Borel $\sigma$-algebra.

The probability measure $\mu_x$ is given by $$\mu_x(E_y)=\frac{1}{(q+1)q^{d(x,y)-1}}.$$ 
The measure $\mu_x$ is invariant under every element of $\textrm{Aut}(\mathcal{T})$ which fixes $x$.
\end{de}

\begin{prop}[Independence of excursions]\label{lem:1} Let $x$ be a vertex of $\mathcal{T}$ which is a lift of $v_0$.

For any $k\ge 1$ and for any $1\le i_1<\cdots<i_k$, $$\mu_x\left(\left\{l\in\mathcal{GT}_x^+|\underset{1\le j\le k}{\max}a_{i_j}{(l)}\le N\right\}\right)=\left(1-\frac{1}{q^{N}}\right)^k.$$ 

\end{prop}
\begin{proof} We prove by induction. Let us denote 
$$A_{i, N}  =\{l\in\mathcal{GT}_x ^+ \,|\,\underset{t_i \le t\le t_{i+1}}{\max} d(v_0,\pi(l_{t}))\le N\}$$ and $A_{i,N}^c$ its complement in $\mathcal{GT}_x^+$ so that
$$\left\{l\in\mathcal{GT}_x^+|\underset{1\le j\le k}{\max}a_{i_j}{(l)}\le N\right\} = \bigcap_{j=1}^k A_{i_j, N}.$$

We first consider the case $k=1$ by computing $\mu_x(A_{i_1, N}^c).$
Let $$V_{i_1}=\{ y \in \pi^{-1}(v_0) : |[xy] \cap v_0 \Gamma|=i_1 \}$$ be the set of starting vertices of $i_1$-th excursions of geodesics.
The geodesic rays with $a_{i_1}^{(l)} > N$ have $l_{t_{i_1}} \in V_{i_1}$ and the $i_1$-th excursion projects to a ray on $\cX$ starting with $v_0v_1 \cdots v_N v_{N+1}$.

The following observation is the keypoint: for each $y \in V_{i_1}$, there exist $q+1$ lifts of $v_1$ which are neighbors of $y$. However, one of them is visited by the geodesic just before it arrives at $y$. Thus, there are exactly $q$ lifts of $v_0v_1$ starting from $y$ not backtracking the geodesic $l_{t_{{i_1}-1}}l_{t_{i_1}}.$ For each of these lifts, there is a unique lift of $v_0 \cdots v_{N+1}$ starting with the lift. Call the endpoints of these lifts $z_j, j=1, \cdots, q$.
%Therefore, there are exactly $q$ distinct lifts $z_i$ of $v_{N+1}$ such that $d(y,z_i)=N+1$ and $y$ is on the geodesic from $x$ to $z_i$.
It follows that
\begin{align}\label{eq:2.1}
\begin{split}
\frac{\mu_x(\{ l : l_{t_{i_1}} =y, \pi(l_{t_{i_1}} \cdots l_{t_{i_1}+N+1}) = v_0 \cdots v_{N+1} \})}{\mu_x(\{ l : l_{t_{i_1}} =y \})}\\
=\frac{\sum_{i=1}^q \mu_x(E_{z_i})}{\mu_x(E_y)}=\frac{q\cdot \frac{1}{(q+1)q^{N+1+d(x,y)-1}}}{\frac{1}{(q+1)q^{d(x,y)-1}}} = \frac{1}{q^N}.
\end{split}
\end{align}

By definition, $l_{t_{i_1}} \in V_{i_1}$ for any $l$. Thus, summing over $y \in V_{i_1}$, we have
$$ \mu_x (A_{i_1, N}^c) = \sum_{ y \in V_{i_1}} \mu_x(\{ l : l_{t_{i_1} }= y\} \cap A_{i_1, N}^c)= \sum \mu_x\{ l: l_{t_{i_1}} = y \} \frac{1}{q^N} =\frac{1}{q^N}  .$$

Now suppose the proposition holds up to $k-1$.
Replacing $y \in V_{i_1}$ by $z \in V_{i_k}$ in the equations \eqref{eq:2.1},
the equation
\begin{equation}\label{eq:2.3}
\frac{\mu_x(\{ l : l_{t_{i_k}} =z \}\cap \bigcap_{j=1}^{k-1} A_{i_j, N}\cap A_{i_k,N})}{\mu_x(\{ l :  l_{t_{i_k}} =z\} \cap\bigcap_{j=1}^{k-1} A_{i_j, N})}=1-\frac{1}{q^N}\end{equation}
holds if and only if both the numerator and the denominator of the left hand side are not zero. Equivalently, $[xz]$ is the beginning of a geodesic in $\bigcap_{j=1}^{k-1} A_{i_j,N}$, i.e. $[xz]$ does not project to a ray starting with $v_0 \cdots v_{N+1}$ on the $\cdots i_{k-1}$-th excursions. 

By induction hypothesis, it follows that
\begin{align*}&\mu_x\left(\left\{l\in\mathcal{GT}_x^+|\underset{1\le j\le k}{\max}a_{i_j}^{(l)} \le N\right\}\right) 
=\mu_x( \bigcap_{j=1}^{k-1} A_{i_j, N}\cap A_{i_k,N}))\\
&= \sum_{z \in V_{i_k}} \mu_x ( \{ l : l_{t_{i_k}} =z \} \cap \bigcap_{j=1}^{k-1} A_{i_j, N}\cap A_{i_k,N})\\
&= (1-\frac{1}{q^N}) \sum_{z \in V_{i_k}} \mu_x (\{ l : l_{t_{i_k}} = z \} \cap \bigcap_{j=1}^{k-1} A_{i_j, N})
=\left(1-\frac{1}{q^N}\right) \mu_x( \bigcap_{j=1}^{k-1} A_{i_j, N})\\
&=\left(1-\frac{1}{q^N}\right) ^k.
\end{align*}
This completes the proof of the proposition.
\end{proof}
Although the proof is lengthy, the main idea of the proof above is that each excursion is independent. We will use this fact again in Section~\ref{sec:4} for more general discrete subgroups.

\begin{prop} For any $x$, 
 $$\lim_{N \to \infty} \mu_x\left(\left\{l\in\mathcal{GT}_x^+|\underset{1\le j\le q^N}{\max}a_{j}{(l)}\le N+y\right\}\right)=e^{-q^y}.$$ 
\end{prop}
\begin{proof} By Lemma~\ref{lem:1}, we have $$\mu_x\left(\left\{l\in\mathcal{GT}_x^+|\underset{1\le j\le n}{\max}a_{j}{(l)}\le N+y\right\}\right)=\left(1-\frac{1}{q^{N+y}}\right)^n.$$ Letting $n=q^N$, and $N \to \infty$, we obtain the proposition.
\end{proof}

We now prove a similar result for bi-infinite geodesics. Note that  \begin{align*}
&\lim_{N\to\infty}(\mu_x\times\mu_x)\left(\left\{l\in\mathcal{GT}_x|\underset{1\le j\le q^N}{\max}a_{j}{(l)}\le N+y\right\}\right)\\
&=\,\lim_{N\to\infty}\mu_x\left(\left\{l\in\mathcal{GT}_x^+|\underset{1\le j\le q^N}{\max}a_{j}{(l)}\le N+y\right\}\right)=e^{-q^y}.
\end{align*}
\begin{prop}\label{prop:2.6}
$$\lim_{N\to\infty}\mu\left(\left\{l\in\mathcal{GT}|\underset{1\le j\le q^N}{\max}a_{j}{(l)}\le N+y\right\}\right)=e^{-q^y}.$$
\end{prop}

\begin{proof}
Choose a lift $x_i$ in $V\mathcal{T}$ of $v_i$. Recall from Definition~\ref{excursion} that $t_1$ is the smallest non-negative integer satisfying $\pi(l_{t_1})=v_0.$
For $i\ne 0$, we have
\begin{align*} &C_0 (\mu_{x_i}\times\mu_{x_i})  \left(\left\{l\in\mathcal{GT}_{x_i}|\underset{1\le j\le q^N}{\max}a_{j}{(l)}\le N+y\right\}\right)\\\displaybreak[0]
=&C_0 \sum_{k=0}^{\infty}(\mu_{x_i}\times\mu_{x_i})\left(\left\{l\in\mathcal{GT}_{x_i}|\underset{1\le j\le q^N}{\max}a_{j}{(l)}\le N+y,t_1=i+2k\right\}\right)\\\displaybreak[0]
=&\sum_{k=0}^{\infty}\mu\left(\left\{[l]\in\Gamma\backslash\mathcal{GT}|\underset{1\le j\le q^N}{\max}a_{j}{(l)}\le N+y,t_1=i+2k\right\}\cap\pi\left(\mathcal{GT}_{x_i}\right)\right)\\\displaybreak[0]
=&\sum_{k=0}^{\infty}\mu\left(\phi^{i+2k}\left[\left\{[l]\in\Gamma\backslash\mathcal{GT}|\underset{1\le j\le q^N}{\max}a_{j}{(l)}\le N+y,t_1=i+2k\right\}\cap\pi\left(\mathcal{GT}_{x_i}\right)\right]\right)\\\displaybreak[0]
=&\mu\left(\left\{[l]\in\Gamma\backslash\mathcal{GT}|\underset{1\le j\le q^N}{\max}a_{j}{(l)}\le N+y\right\}\cap\pi\left(\mathcal{GT}_{x}\right) \cup_{k=0}^{\infty} \{ l \in \Gamma \backslash \mathcal{GT} : t_1 = i + 2k\}\right)\\\displaybreak[0]
=&C_0 (\mu_{x}\times\mu_{x})\left(\left\{l\in\mathcal{GT}_{x}|\underset{1\le j\le q^N}{\max}a_{j}{(l)}\le N+y\right\}\right).
\end{align*} %\color{blue} I considered a fixed $x_i$ not all the vertices, because it is simpler and I think it still works. Please check whether this is right. \color{black} 
The $\phi$-invariance of $\mu$ gives the third equality.
\end{proof}

Using Proposition~\ref{prop:2.6}, we prove the main theorem for the rays of Nagao type. Recall that $h_T^{(l)}=\underset{0\le t\le T}{\max}d(\pi(l_t),v_0)$ and that $t_n$ is the starting time of the $n$-th excursion of $l$.

For each geodesic $l$, let $S_n^{(l)}=2(a_1^{(l)}+\cdots +a_n^{(l)})= t_{n+1}-t_1$ be the total time of the first $n$ excursions and $T_n$ its expectation with respect to $\mu$.
Note that 
\begin{align*}T_n &= \mathbb{E}_\mu \left(\sum_{i=1}^n2a_i\right)=2n\mathbb{E}_\mu(a_1)=2n\sum_{k=1}^{\infty}k\mu(a_1=k)
=\sum_{k=1}^{\infty}\frac{2nk(q-1)}{q^k}\\&=\frac{2qn}{q-1}.
\end{align*}
\begin{thm}\label{main}
We have
$$\lim_{T\to\infty}\mu\left(\left\{[l]\in\Gamma\backslash\mathcal{GT} : h_{T}^{(l)}\le \log_q\left(\frac{T(q-1)}{2q}\right)+y\right\}\right)=e^{-1/q^y}.$$
\end{thm}
\begin{proof}

%If $\pi(x)=v_j$ $(j\ne0)$, then the expectation of $t_n$ is
%\begin{align*}\mathbb{E}\left(t_1+\sum_{i=1}^n2a_i\right)&=\mathbb{E}(t_1)+2n\mathbb{E}(a_1)=\sum_{k=1}^{\infty}k\mu(t_1=k)+2n\sum_{k=1}^\infty k\mu(a_1=k)\\
%&=\sum_{k=1}^{\infty}\frac{(j+2k-2)(q-1)}{q^k}+\frac{2nq}{q-1} \\
%&=\frac{2nq+(2j-2)(q-1)}{q-1}.
%\end{align*}

By the law of large numbers, we have $\frac{S_n^{(l)}-T_n}{n}\to 0$ for $\mu$-almost every $l\in\mathcal{GT}$. Moreover, if we denote by $B_{n,C}$ the set $\{l\in\mathcal{GT}\colon |S_n^{(l)}-T_n|\le C\sqrt{n}\}$, then by the central limit theorem of the shift map, for any $\epsilon>0$, there exists $C>0$ such that $$\mu_x(B_{n,C})>1-\epsilon$$ holds for all $n\ge 1$.

Let $A_{T,N+y}=\{l\in\mathcal{GT}\,|\,\underset{1\le t\le T}{\max} d(v_0,\pi(l_t))\le N+y\}$. Note that $$A_{T_{q^N}+Cq^{N/2},N+y}\subseteq A_{T_{q^N},N+y}\subseteq A_{T_{q^N}-Cq^{N/2},N+y}.$$
Therefore,
%$$\mu(A_{T_{q^N}+Cq^{N/2},N+y})\le\mu( A_{T_{q^N},N+y})\le\mu(A_{T_{q^N}-Cq^{N/2},N+y})$$ and
$$\mu(A_{T_{q^N}+Cq^{N/2},N+y}\cap B_{n,C})\le\mu( A_{S_{q^N}^{(l)},N+y}\cap B_{n,C})\le\mu(A_{T_{q^N}-Cq^{N/2},N+y}\cap B_{n,C}).$$
Meanwhile, 
\begin{align*}
&\mu(A_{T_{q^N}- Cq^{N/2},N+y})-\mu(A_{T_{q^N}+Cq^{N/2},N+y})\\
&=\mu\left(\left\{l\in\mathcal{GT}\colon\underset{T_{q^N}-Cq^{N/2}\le t\le T_{q^N}+Cq^{N/2}}{\max}d(v_0,\pi(l_t))>N+y\right\}\right)\\
&=\mu\left(\left\{l\in\mathcal{GT}\colon\underset{0\le t\le 2Cq^{N/2}}{\max} d(v_0,\pi(l_t))> N+y\right\}\right) \, (\mu\textrm{ is }\phi\textrm{-invariant})\\ \displaybreak[0]
&\le\mu\left(\left\{l\in\mathcal{GT}\colon\underset{0\le t\le \frac{2q^{(2N/3)+1}}{q-1}-Cq^{N/3}}{\max} d(v_0,\pi(l_t))> N+y\right\}\right) \\&\qquad\qquad\qquad\left(\frac{2q^{(2N/3)+1}}{q-1}-Cq^{N/3}\ge 2Cq^{N/2}\textrm{ for sufficiently large }N\right)\\
&\le\mu\left(\left\{l\in\mathcal{GT}\cap B_{n,C}\colon\underset{0\le t\le S_{q^{2N/3}}^{(l)}}{\max} d(v_0,\pi(l_t))> N+y\right\}\right)+\epsilon.
\end{align*}

Hence, for any given $\epsilon>0$, there exists $M>0$ such that $$\mu(A_{S_{q^N}^{(l)}},N+y)-2\epsilon\le\mu(A_{T_{q^N}},N+y)\le\mu(A_{S_{q^N}^{(l)}},N+y)+2\epsilon$$ holds for all $N\ge M$. By Proposition~\ref{prop:2.6}, we have
 $$\lim_{N \to \infty} \mu\left(\left\{l\in\mathcal{GT}\colon\underset{0\le t\le \frac{2q^{N+1}}{q-1}}{\max}d(v_0, \pi(l_t))\le N+y\right\}\right)=e^{-q^y}$$ which completes the proof.
\end{proof}

\section{Extreme value distribution for geometrically finite quotient} \label{sec:4}
In this section, we prove extreme value distribution for geometrically finite quotient graphs of regular trees using Markov chain on the compact part and the extreme value distribution result for each ray proved in the previous section.

We remark that an alternative approach might be to use general extreme value theorem \cite{Fr} using the $\phi$-mixing property, i.e., the error term of mixing $|\mu(A\cap T^{-n}B)-\mu(A)\mu(B)|$ is bounded by the measures of the sets $A$ and $B$) of the measure-preserving transformation $T$, which is not available here. Note that exponential mixing is known \cite{BPP} based on a result of Young \cite{Yo} (see also the paper by the first author \cite{k18}). %We might persue this direction for general lattices (whose cusps are not necessarily rays of ``Nagao type").

Let us first fix some notations on the quotient graph.
Given a $(q+1)$-regular tree $\mathcal{T}$, let $\textrm{Aut}(\mathcal{T})$ be the group of automorphisms of $\mathcal{T}$ and $\Gamma$ be a geometrically finite discrete subgroup of $\textrm{Aut}(\mathcal{T})$ (see Section~\ref{sec:1}). There are finite edge-indexed rays $C_1,\cdots,C_k$  and a finite edge-indexed graph $D$ such that
\begin{enumerate}\setlength\itemsep{-\parsep} 
\item $V(\Gamma\backslash\mathcal{T}_{\min})=VD\cup VC_1\cup\cdots\cup VC_k$.
\item $|VD\cap VC_j|=1$ and $VC_i\cap VC_j=\phi$ if $i\ne j$.
\item Each $C_i$ is a Nagao ray of index $(1,q,1, q,\ldots,)$, i.e., $i(\overline{e_n})=q,i(e_n)=1$ for all $n\ge 0$.
\end{enumerate}
Fix such $C_j$ and $D$. Let us denote by $v_{i,0}$ the unique element of $VD\cap VC_i$ $(1\le i\le k)$. Then, we obtain the following figure.
\
\vspace{0.5em}
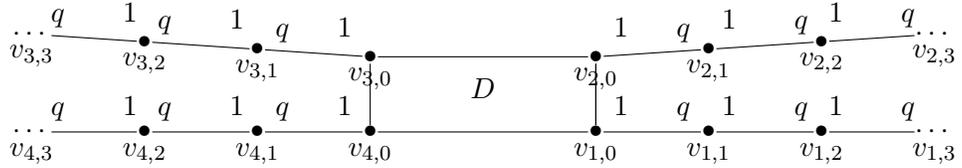
\begin{figure}[H]\label{fig:3}
\begin{center}
\begin{tikzpicture}[every loop/.style={}]
  \tikzstyle{every node}=[inner sep=0pt]
  \node (0) {} node [above=4pt] at (0.5,1.5) {};
  \node (2) at (1.5,0) {$\bullet$} node [below=4pt] at (1.5,0) {$v_{1,0}$}; 
  \node (4) at (3,0) {$\bullet$}node [below=4pt] at (3,0) {$v_{1,1}$}; 
  \node (6) at (4.5,0) {$\bullet$}node [below=4pt] at (4.5,0) {$v_{1,2}$}; 
  \node (8) at (6,0) {$\cdots$}node [below=4pt] at (6,0) {$v_{1,3}$}; 
%  \node (10) at (6,0) {$\cdots$}node [below=4pt] at (6,0) {$$}; 
  
  \path[-] %(0) edge node [above=4pt] { a \,\, $q$} (2)
 (2) edge node [above=4pt] {$1$ \quad \,\,$q$} (4)
(4) edge node [above=4pt] {$1$ \quad \,\,\,\,$q$} (6)
 (6) edge node [above=4pt] {$1$ \quad \,\,\,\,$q$} (8);
 %(8) edge node [above=4pt] {$1$ \quad \,\,\,\,$q$} (10);
 
%   \node (10) {$\bullet$} node [below=4pt] at (0,1) {$$};
  \node (20) at (1.5,1) {$\bullet$} node [below=4pt] at (1.5,1) {$v_{2,0}$}; 
  \node (40) at (3,1.1) {$\bullet$}node [below=4pt] at (3,1.1) {$v_{2,1}$}; 
  \node (60) at (4.5,1.2) {$\bullet$}node [below=4pt] at (4.5,1.2) {$v_{2,2}$}; 
  \node (80) at (6,1.3) {$\cdots$}node [below=4pt] at (6,1.3) {$v_{2,3}$}; 
 % \node (100) at (7.5,1.4) {$\cdots$}node [below=4pt] at (7.5,1.4) {$$}; 
  
  \path[-] %(0) edge node [above=4pt] { } (10)
  (2) edge node [above=4pt] {  $$} (20)
 (20) edge node [above=4pt] {$1$ \quad \,\,$q$} (40)
(40) edge node [above=4pt] {$1$ \quad \,\,\,\,$q$} (60)
 (60) edge node [above=4pt] {$1$ \quad \,\,\,\,$q$} (80);
% (80) edge node [above=4pt] {$1$ \quad \,\,\,\,$q$} (100);

  % \node (1) {$\bullet$} node [below=4pt] at (0,0) {$$};
  \node (3) at (-1.5,0) {$\bullet$} node [below=4pt] at (-1.5,0) {$v_{4,0}$}; 
  \node (5) at (-3,0) {$\bullet$}node [below=4pt] at (-3,0) {$v_{4,1}$}; 
  \node (7) at (-4.5,0) {$\bullet$}node [below=4pt] at (-4.5,0) {$v_{4,2}$}; 
  \node (9) at (-6,0) {$\cdots$}node [below=4pt] at (-6,0) {$v_{4,3}$}; 
%  \node (11) at (-7.5,0) {$\cdots$}node [below=4pt] at (-7.5,0) {$$}; 

   \path[-] 
(3) edge node [above=4pt] {$q$ \quad \,\,$1$} (5)
(5) edge node [above=4pt] {$q$ \quad \,\,\,\,$1$} (7)
 (7) edge node [above=4pt] {$q$ \quad \,\,\,\,$1$} (9);
% (9) edge node [above=4pt] {$q$ \quad \,\,\,\,$1$} (11);

%   \node %(10) {$\bullet$} node [below=4pt] at (0,1) {$$};
  \node (30) at (-1.5,1) {$\bullet$} node [below=4pt] at (-1.5,1) {$v_{3,0}$}; 
  \node (50) at (-3,1.1) {$\bullet$}node [below=4pt] at (-3,1.1) {$v_{3,1}$}; 
  \node (70) at (-4.5,1.2) {$\bullet$}node [below=4pt] at (-4.5,1.2) {$v_{3,2}$}; 
  \node (90) at (-6,1.3) {$\cdots$}node [below=4pt] at (-6,1.3) {$v_{3,3}$}; 
 % \node (110) at (-7.5,1.4) {$\cdots$}node [below=4pt] at (-7.5,1.4) {$$}; 

   \path[-] 
  % (0) edge node [above=4pt] {??} (10)
  (3) edge node [above=4pt] {$ $} (30)
  (3) edge node [above=4pt] {$ $} (2)
  (20) edge node [above=4pt] {$ $} (2)
  (30) edge node [below=8pt] {$D$} (20) 
 (30) edge node [above=4pt] {$q$ \quad \,\,$1$} (50)
(50) edge node [above=4pt] {$q$ \quad \,\,\,\,$1$} (70)
 (70) edge node [above=4pt] {$q$ \quad \,\,\,\,$1$} (90);
% (90) edge node [above=4pt] {$q$ \quad \,\,\,\,$1$} (110); 
\end{tikzpicture}
%\vspace{0.5em}
\caption{The quotient graph of a geometrically finite subgroup}
\end{center}
\end{figure}

\begin{de} Let $l \in \mathcal{GT}$. 
We can write $$\pi^{-1}(\{v_{1,0},\ldots,v_{k,0}\})=\{\ldots,l_{t_{-1}},l_{t_0},l_{t_1},l_{t_2},\ldots,\}$$ with $t_1$ be the smallest positive time when $l$ leaves the compact part. Note that $\pi(l_{t_{2n-1}})=\pi(l_{t_{2n}})$. The sequence of vertices $l_{t_{2n-1}}l_{t_{2n-1}+1}\cdots l_{t_{2n}}$ is called the $n$\emph{-th excursion of }$l$.
\end{de}
Comparing with Definition~\ref{excursion}, note that the starting time of the $n$-th excursion is now $t_{2n-1}$.
As explained in Definition~\ref{excursion} (2), any $n$-th excursion of geodesic projects to $v_{i,0}\cdots v_{i,m}v_{i,m-1}\cdots v_{i,1}v_{i,0}$ for some $i$ and $m$. We call such $m$ the \emph{height} of $n$-th excursion of $l$ and denote it by $a_n{(l)}$.

Recall that 
$ h_T^{(l)} = \max_{0 \leq t \leq T} d(D, l(t))$ and the Markov chain $(\mathcal{S},p_{ij})$ is given by \begin{align}\label{MCdef}\mathcal{S}=E(\Gamma\backslash\!\backslash\mathcal{T}),\quad p_{ij}=p_{e_ie_j}=\frac{\lambda([e_i,e_j])}{\lambda([e_i])}.\end{align} The stationary distribution is given by 
$ \pi_j=\lambda([e_j]).$
Recall also that $\mu_x$ is the Patterson density for $\Gamma$ based at $x$ (Definition~\ref{def:3.1}) and $\mu$ is the Gibbs measure constructed in Definition~\ref{def:BM}. 
\begin{lem}\label{lem:2} If $\Gamma$ is non-elementary and $\Gamma\backslash\mathcal{T}$ has at least one Nagao ray, then $\delta>\frac{1}{2}\log q$ and we have
$$\overline{\mu}\left(\left\{[l]\in\Gamma_f\backslash\mathcal{GT}\,|\,a_n{(l)}\le N\right\}\right)=1-\frac{q^N}{e^{2\delta N}}.$$
\end{lem}
\begin{proof} The proof is verbatim to the proof of Proposition~\ref{lem:1}, except that we need to obtain the general version of \eqref{eq:2.1}.
For $j=0, \cdots, m$, let $x_j\in V\mathcal{T}$ be the vertices satisfying $\pi(x_j)=v_{i,j}$ and $x_j,x_{j+1}$ are adjacent. For $j=1, \cdots, m$, let $f_j\in E\mathcal{T}$ such that $\partial_0f_j=x_{j-1}$ and $\partial_1f_j=x_j$. 
We need to show that for any integer $N >0$,
 $$\displaystyle\frac{\mu_{x_{i}}(\mathcal{O}({f_{i+N}}))}{\mu_{x_i}(\mathcal{O}({f_{i}}))}=\left(\frac{q}{e^{2\delta}}\right)^N.$$
 Let $v_{i,0}\cdots v_{i,m}v_{i,m-1}\cdots v_{i,1}v_{i,0}$ be the projection of $n$-th excursion of some geodesic $l\in\mathcal{GT}$ under $\pi$. 
 %Recall that $(\Gamma_f\backslash \mathcal{GT},\phi,\overline{\mu})\textrm{ \ and \ }(X_{(A,i)},\sigma,\lambda)$ are isomorphic (as we mentioned in \ref{2}). 

 Let
 $\alpha_j=\mu_{x_{j}}(\mathcal{O}(\overline{f_{j}}))$. (Note that this does not depend on the choice of $x_j$). Since $\Gamma$ is non-elementary, it follows that $\mu$ has no atoms and hence $\alpha_j\ne 0$. The conformal property of $\mu$ implies that $\mu_{x_j}(\mathcal{O}(\overline{f_j}))=\mu_{x_{j+1}}(\mathcal{O}(\overline{f_j}))e^\delta$. Since there are $q$ neighbors of $x_{j}$ which projects to $v_{i,j}$, and $\Gamma_{x_j}$ acts transitively on these neighbors, we have
 \begin{equation}\label{eq:4.2}
 \alpha_{j+1} =  
 q \mu_{x_{j}+1}(\mathcal{O}(\overline{f_{j}})) = q \mu_{x_{j}}(\mathcal{O}(\overline{f_{j}})) e^{-\delta}= q \alpha_j e^{-\delta}. %\mu_{x_{j+1}}(\mathcal{O}(\overline{f_{j+1}})) =
 \end{equation}

%Thus, we have
%\begin{align*}
%\mu_{x_j}(\mathcal{O}(f_{j+1}))
%&=(q-1)\mu_{x_{j+1}}(\mathcal{O}(\overline{f_{j+1}}))e^{-\delta}+(q-1)\mu_{x_{j+2}}(\mathcal{O}(\overline{f_{j+2}}))e^{-2\delta}+\cdots \\
%%&=(q-1)qx_{i}e^{-2\delta}+(q-1)q^2x_ie^{-4\delta}+\cdots \\
%&=(q-1)\alpha_j\sum_{n=1}^{\infty}q^{n}e^{-2n\delta}
%\end{align*} and
%\begin{align*}
%\mu_{x_j}(\mathcal{O}(f_{j}))
%&=(q-1)\mu_{x_{j}}(\mathcal{O}(\overline{f_{j}}))+(q-1)\mu_{x_{j+1}}(\mathcal{O}(\overline{f_{j+1}}))e^{-\delta}+\cdots \\
%&=(q-1)\alpha_{j}+(q-1)q\alpha_{j}e^{-2\delta}+\cdots \\
%&=(q-1)\alpha_j\sum_{n=0}^{\infty}q^{n}e^{-2n\delta}.
%\end{align*}
Let $e_i$ be the edge given by $\partial_0e_i=v_{i-1}$ and $\partial_1e_i=v_{i}$. Let us decompose the shadow $\mathcal O (f_{j+N})$ into countable disjoint union of sets:
$\mathcal O (f_{j+N})$ is the union of $(q-1)$ shadows $O(g_0)$ of lifts $g_0$ of $\overline{e_{j+N}}$ adjacent to $f_{j+N}$ and the shadow $\mathcal O(f_{j+N+1})$.

The shadow $\mathcal O(f_{j+N+1})$ is in turn the union of $q-1$ shadows $\mathcal O(g_1)$ of lifts $g_1$ of $\overline{e_{j+N+1}}$ adjacent to $f_{j+N+1}$ and the shadow $\mathcal O(f_{j+N+2})$. We repeat this decomposition. For any $l \geq 0$, we obtain $q-1$ shadows $\mathcal O(g_l)$'s such that
$$\mu_{x_j}(\mathcal O(g_l))= \mu_{x_j}( \mathcal O(\overline{f_{j+N+l}}) )e^{-2(N+l)\delta}=\alpha_j q^{N+l} e^{-2(N+l)\delta}$$
by the conformal property and \eqref{eq:4.2}

Therefore, for any $j,N\ge 0$, we have $$\mu_{x_j}(\mathcal{O}(f_{j+N}))=(q-1)\alpha_j\sum_{n=N}^\infty q^ne^{-2n\delta}=\frac{(q-1)\alpha_j\left(\frac{q}{e^{2\delta}}\right)^{N}}{1-\frac{q}{e^{2\delta}}}.$$
Since $\mu_{x_j}(\mathcal{O}(f_j))<\infty$, the above series must converge thus we have $\delta_{\Gamma}>\frac{1}{2}\log q$
and $\displaystyle\frac{\mu_{x_{j}}(\mathcal{O}_{f_{j+N}})}{\mu_{x_j}(\mathcal{O}_{f_{j}})}=\left(\frac{q}{e^{2\delta}}\right)^N$.
%\seon{Change $v_i$ to vertices in the tree.}
\end{proof}
Therefore, by independence of excursions (similar to the proof of Proposition~\ref{lem:1}), we obtain a limiting \emph{Galambos type formula}. 
$$\overline{\mu}\left(\left\{[l]\in\Gamma_f\backslash\mathcal{GT}|\underset{1\le j\le k}{\max}a_{j}{(l)}\le N+y\right\}\right)=\left(1-\frac{q^{N+y}}{e^{2\delta(N+y)}}\right)^k. $$

Given a geodesic $l$ in $\mathcal{GT}$, let us denote by $C{(l)}$ the expectation $$\lim_{N\to\infty}\frac{1}{N}\sum_{n=1}^{N}\left(l_{t_{2n+1}}-l_{t_{2n}}\right)$$ of the difference between the starting time of the $(n+1)$-th excursion time and the ending time of the $n$-th excursion of $l$ (the time living in the compact part) over $n\in\mathbb{Z}_{>0}$. Note that this does not depend on the choice of representative in $\Gamma\backslash \mathcal{GT}$ and the limit exists for $\mu$-almost every $[l]\in\Gamma\backslash\mathcal{GT}$.

\begin{lem}Let $C_\Gamma=\int_{\Gamma\backslash\mathcal{GT}}C{([l])}d\mu$.
The expectation with respect to $\mu$ of the time of $n$ excursions $t_{2n}-t_1$ over $\Gamma\backslash\mathcal{GT}$ is $$\left(\frac{2e^{2\delta}}{e^{2\delta}-q}+C_\Gamma\right)n.$$

\end{lem}
\begin{proof}
Since the Markov chain associated to the compact part is finite, it is positive recurrent (Chapter 10, \cite{MT}).
Hence, the constant $C_\Gamma$ is finite and depends only on the structure of quotient graph $\Gamma\backslash\mathcal{T}$ and the choice of the compact part $D$. %Hence, there exists a stationary distribution. %$$C_\Gamma n+\sum_{j=1}^\infty \frac{2nj(e^{2\delta}-q)q^{j-1}}{e^{2\delta j}}$$

The expectation with respect to $\mu$ of $t_{2n}-t_1$ of $l$ is
\begin{align*}&\mathbb{E}_\mu \left(\sum_{i=1}^n(2a_i^{(l)}+C^{(l)})\right)=n(2\mathbb{E}_\mu(a_1)+\mathbb{E}_\mu C^{(l)})=C_\Gamma n+2n\sum_{k=1}^{\infty}k\mu(a_1=k)\\
=&C_\Gamma n+\sum_{k=1}^{\infty}\frac{2nk(\frac{e^{2\delta}}{q}-1)}{(\frac{e^{2\delta}}{q})^k}=C_\Gamma n+\sum_{k=1}^\infty \frac{2nk(e^{2\delta}-q)q^{k-1}}{e^{2\delta k}}=\left(\frac{2e^{2\delta}}{e^{2\delta}-q}+C_\Gamma\right)n.
\end{align*} This completes the proof of the lemma.
\end{proof}

By the similar argument deriving Theorem~\ref{main} from Proposition~\ref{prop:2.6}, we finally have that
$$\lim_{N\to\infty}\mu\left(\left\{[l]\in\Gamma\backslash\mathcal{GT}|\underset{{0\le t\le T}}{\max}h_T^{(l)}\le N+y\right\}\right)=e^{-q^y/e^{2\delta y}}$$
with $$T=\left(\frac{2e^{2\delta}}{e^{2\delta}-q}+C_\Gamma\right)\frac{e^{2\delta N}}{q^N}.$$
Therefore, $$\lim_{T\to\infty}\mu\left(\left\{[l]\in\Gamma\backslash\mathcal{GT}|h_T^{(l)}\le \log_{e^{2\delta}/q}\left(\frac{T(e^{2\delta}-q)}{2e^{2\delta}-C_\Gamma(e^{2\delta}-q)}\right)+y\right\}\right)=e^{-q^y/e^{2\delta y}}.$$

\textbf{Acknowledgement:} We would like to thank J. Athreya, A. Ghosh, and M. Kirsebom for helpful discussions and explaining the problem and their work. The first author is supported by Samsung STF Project no. SSTF-BA1601-03.

\par

\end{document}